\documentclass{article}
\usepackage[utf8]{inputenc}

\usepackage{amsmath,amsthm,amssymb,amsfonts}
\usepackage{verbatim}
\usepackage{array}

\usepackage{hyperref}

\usepackage{stmaryrd} 
\usepackage{hyperref}

\usepackage[colorinlistoftodos]{todonotes}

\newtheorem{thm}{Theorem}
\newtheorem{lem}[thm]{Lemma}
\newtheorem{prop}[thm]{Proposition}
\newtheorem{exl}[thm]{Example}
\newtheorem{cor}[thm]{Corollary}

\newtheorem{defn}[thm]{Definition}
\theoremstyle{remark}
\newtheorem{remark}[thm]{Remark}
\theoremstyle{defn}

\numberwithin{thm}{section}

\newcommand{\adj}{\leftrightarrow}
\newcommand{\adjeq}{\leftrightarroweq}

\DeclareMathOperator{\id}{id}

\DeclareMathOperator{\Fix}{Fix}
\newcommand{\Z}{\mathbb{Z}}

\title{Freezing Sets for Arbitrary Digital Dimension}

\author{Laurence Boxer 
        \thanks{Department of Computer and Information Sciences,
        Niagara University, NY 14109, USA; and \newline
        Department of Computer Science and Engineering,
        State University of New York at Buffalo\newline
        email: boxer@niagara.edu}
}

\begin{document}
\date{ }
\maketitle{}

\begin{abstract}
    We show how to obtain freezing sets for digital images
    in~$\Z^n$ for arbitrary~$n$, using the $c_1$ and $c_n$
    adjacencies.
\end{abstract}

\section{Introduction}
The study of freezing sets is part of the fixed point theory
of digital topology. Freezing sets were introduced in 
in~\cite{BxFpSets2} and studied in subsequent papers
including~\cite{BxConvex,subsetsAnd,consequences}. These papers
focus mostly on digital images in $\Z^2$.

In the current paper, we obtain results for freezing sets in $\Z^n$,
for arbitrary~$n$. We show that given a finite connected digital image
$X \subset \Z^n$, if we use the $c_1$ or $c_n$ adjacency
and $X$ is decomposed into a union of cubes~$K_i$,
then we can construct a freezing set for $X$ from those of the $K_i$.

\section{Preliminaries}
\subsection{Adjacencies}
Much of this section is quoted or paraphrased from~\cite{bs20}.

A digital image is a pair $(X,\kappa)$ where
$X \subset \Z^n$ for some $n$ and $\kappa$ is
an adjacency on $X$. Thus, $(X,\kappa)$ is a graph
with $X$ for the vertex set and $\kappa$ 
determining the edge set. Usually, $X$ is finite,
although there are papers that consider infinite $X$. 
Usually, adjacency reflects some type of
``closeness" in $\Z^n$ of the adjacent points.
When these ``usual" conditions are satisfied, one
may consider the digital image as a model of a
black-and-white ``real world" digital image in which
the black points (foreground) are 
the members of $X$ and the white points 
(background) are members of $\Z^n \setminus X$.

We write $x \adj_{\kappa} y$, or $x \adj y$ when
$\kappa$ is understood or when it is unnecessary to
mention $\kappa$, to indicate that $x$ and
$y$ are $\kappa$-adjacent. Notations 
$x \adjeq_{\kappa} y$, or $x \adjeq y$ when
$\kappa$ is understood, indicate that $x$ and
$y$ are $\kappa$-adjacent or are equal.

The most commonly used adjacencies are the
$c_u$ adjacencies, defined as follows.
Let $X \subset \Z^n$ and let $u \in \Z$,
$1 \le u \le n$. Then for points
\[x=(x_1, \ldots, x_n) \neq (y_1,\ldots,y_n)=y\]
we have $x \adj_{c_u} y$ if and only if
\begin{itemize}
    \item for at most $u$ indices $i$ we have
          $|x_i - y_i| = 1$, and
    \item for all indices $j$, $|x_j - y_j| \neq 1$
          implies $x_j=y_j$.
\end{itemize}

The $c_u$-adjacencies are often denoted by the
number of adjacent points a point can have in the
adjacency. E.g.,
\begin{itemize}
\item in $\Z$, $c_1$-adjacency is 2-adjacency;
\item in $\Z^2$, $c_1$-adjacency is 4-adjacency and
      $c_2$-adjacency is 8-adjacency;
\item in $\Z^3$, $c_1$-adjacency is 8-adjacency,
      $c_2$-adjacency is 18-adjacency, and 
      $c_3$-adjacency is 26-adjacency.
\end{itemize}

In this paper, we mostly use the $c_1$ and $c_n$ adjacencies.

When $(X,\kappa)$ is understood to be a digital image under discussion,
we use the following notations. For $x \in X$,
\[ N(x) = \{y \in X \,| \, y \adj_{\kappa} x\},
\]
\[ N^*(x) = \{y \in X \,| \, y \adjeq_{\kappa} x\} = N(x) \cup \{x\}.
\]

\begin{defn}
{\rm \cite{Ros79}}
Let $X \subset \Z^n$. The {\em boundary of } $X$ is
\[ Bd(X) = \{ \, x \in X \mid \mbox{there exists } y \in \Z^n \setminus X
              \mbox{ such that } x \adj_{c_1} y \, \}.
\]
\end{defn}

\subsection{Digitally continuous functions}
Much of this section is quoted or paraphrased from~\cite{bs20}.

We denote by $\id$ or $\id_X$ the
identity map $\id(x)=x$ for all $x \in X$.

\begin{defn}
{\rm \cite{Rosenfeld, Bx99}}
Let $(X,\kappa)$ and $(Y,\lambda)$ be digital
images. A function $f: X \to Y$ is 
{\em $(\kappa,\lambda)$-continuous}, or
{\em digitally continuous} or just {\em continuous} when $\kappa$ and
$\lambda$ are understood, if for every
$\kappa$-connected subset $X'$ of $X$,
$f(X')$ is a $\lambda$-connected subset of $Y$.
If $(X,\kappa)=(Y,\lambda)$, we say a function
is {\em $\kappa$-continuous} to abbreviate
``$(\kappa,\kappa)$-continuous."
\end{defn}

\begin{thm}
{\rm \cite{Bx99}}
A function $f: X \to Y$ between digital images
$(X,\kappa)$ and $(Y,\lambda)$ is
$(\kappa,\lambda)$-continuous if and only if for
every $x,y \in X$, if $x \adj_{\kappa} y$ then
$f(x) \adjeq_{\lambda} f(y)$.
\end{thm}

\begin{thm}
\label{composition}
{\rm \cite{Bx99}}
Let $f: (X, \kappa) \to (Y, \lambda)$ and
$g: (Y, \lambda) \to (Z, \mu)$ be continuous 
functions between digital images. Then
$g \circ f: (X, \kappa) \to (Z, \mu)$ is continuous.
\end{thm}

A {\em $\kappa$-path} is a continuous
function $r: ([0,m]_{\Z},c_1) \to (X,\kappa)$. 

For a digital image $(X,\kappa)$, we use the notation
\[ C(X,\kappa) = \{f: X \to X \, | \,
   f \mbox{ is continuous}\}.
\]

A function $f: (X,\kappa) \to (Y,\lambda)$ is
an {\em isomorphism} (called a {\em homeomorphism}
in~\cite{Bx94}) if $f$ is a continuous bijection
such that $f^{-1}$ is continuous.

For $X \in \Z^n$, the {\em projection to the $i^{th}$ coordinate}
is the function $p_i: X \to \Z$ defined by
\[ p_i(x_1, \ldots, x_n) = x_i.
\]

A {\em (digital) line segment} in $(X,\kappa)$ is a set 
$S = f([0,m]_{\Z})$, where $f$ is a digital path,
such that the points of $S$ are collinear;
$S$ is {\em axis parallel}
if for all but one of the indices~$i$,
$p_i \circ f$ is a constant function.

\subsection{Cube terminology}
\label{cubeSec}
Let $Y = \Pi_{i=1}^n [a_i,b_i]_{\Z}$, where $b_i > a_i$.

If for $1 \le j \le n$ there are exactly $j$ indices~$i$ such that
$b_i > a_i$ (equivalently, exactly $n-j$ indices~$i$ such that
$b_i = a_i$), we call $Y$ a {\em $j$-dimensional cube} or a
{\em $j$-cube}.

A $j$-cube $K$ in $Y$ such that 
\begin{itemize}
    \item for $j$ indices~$i$, $p_i(K) = [a_i,b_i]_{\Z}$ and
    \item for all other indices~$i$, $p_i(K) = \{a_i\}$ or 
          $p_i(K) = \{b_i\}$,
\end{itemize}
is a {\em face} or a {\em $j$-face} of $Y$.

A {\em corner} of $Y$ is any of the points of $\Pi_{i=1}^n \{a_i,b_i\}$.
An {\em edge} of $Y$ is an axis-parallel digital line segment
joining two corners of $Y$.

\section{Tools for determining fixed point sets}
\begin{defn}
{\rm \cite{BxFpSets2}}
\label{freezeDef}
Let $(X,\kappa)$ be a
digital image. We say
$A \subset X$ is a 
{\em freezing set for $X$}
if given $g \in C(X,\kappa)$,
$A \subset \Fix(g)$ implies
$g=\id_X$.
\end{defn}

\begin{thm}
{\rm \cite{BxFpSets2}}
\label{freezeInvariant}
Let $A$ be a freezing set for the digital image $(X,\kappa)$ and let
$F: (X,\kappa) \to (Y,\lambda)$ be an isomorphism. Then $F(A)$ is
a freezing set for $(Y,\lambda)$.
\end{thm}

The following are useful for determining fixed 
point and freezing sets.

\begin{prop}
{\rm (Corollary 8.4 of~\cite{bs20})}
\label{uniqueShortestProp}
Let $(X,\kappa)$ be a digital image and
$f \in C(X,\kappa)$. Suppose
$x,x' \in \Fix(f)$ are such that
there is a unique shortest
$\kappa$-path $P$ in~$X$ from $x$ 
to $x'$. Then $P \subseteq \Fix(f)$.
\end{prop}

Lemma~\ref{c1pulling}, below,
\begin{quote}
$\ldots$ can
be interpreted to say that
in a $c_u$-adjacency,
a continuous function that
moves a point~$p$ also moves
a point that is ``behind"
$p$. E.g., in $\Z^2$, if $q$ and $q'$ are
$c_1$- or $c_2$-adjacent with $q$
left, right, above, or below $q'$, and a
continuous function $f$ moves $q$ to the left,
right, higher, or lower, respectively, then
$f$ also moves $q'$ to the left,
right, higher, or lower, respectively~\cite{BxFpSets2}.
\end{quote}

\begin{lem}
\label{c1pulling}
{\rm ~\cite{BxFpSets2}}
Let $(X,c_u)\subset \Z^n$ be a digital image, 
$1 \le u \le n$. Let $q, q' \in X$ be such that
$q \adj_{c_u} q'$.
Let $f \in C(X,c_u)$.
\begin{enumerate}
    \item If $p_i(f(q)) > p_i(q) > p_i(q')$
          then $p_i(f(q')) > p_i(q')$.
    \item If $p_i(f(q)) < p_i(q) < p_i(q')$
          then $p_i(f(q')) < p_i(q')$.
\end{enumerate}
\end{lem}

\begin{defn}
{\rm \cite{BxConvex}}
\label{closeNbrDef}
{\rm Let $(X,\kappa)$ be a digital image. Let
$p,q \in X$ such that 
\[ N(X,p,\kappa) \subset N^*(X,q,\kappa).
\]
Then $q$ is a {\em close $\kappa$-neighbor}
of $p$.
}
\end{defn}

\begin{lem}
{\rm \cite{bs20,BxConvex}}
\label{closeNbr}
Let $(X,\kappa)$ be a digital image. Let
$p,q \in X$ such that $q$ is a close 
$\kappa$-neighbor of $p$. Then $p$ belongs to every
freezing set of $(X,\kappa)$.
\end{lem}

\begin{thm}
\label{closeCorners}
Let $Y = \Pi_{i=1}^3 [a_i,b_i]_{\Z}$ be
      such that $b_i > a_i +1$ for all $i$.
Let $A = \Pi_{i=1}^3 \{a_i,b_i\}$.
Then $A$ is a subset of every freezing set for $(X,c_3)$.
\end{thm}

\begin{proof}
By Theorem~\ref{freezeInvariant}, we may assume 
$a_i=0$ for all $i$, so
$A = \Pi_{i=1}^3 \{0,b_i\}$. It is easily seen that
every $a \in A$ has a close neighbor in $X$, namely the
unique member of $X$ that differs from $a$ by 1 in every
coordinate. Therefore, by Lemma~\ref{closeNbr},
$A$ is a subset of every freezing set for $(X,c_u)$. 
\end{proof}

\section{$c_1$-Freezing sets for cubes}
The following is presented as 
Theorem~5.11 of~\cite{BxFpSets2}.
However, there is an error in the argument given
in \cite{BxFpSets2} for the proof of the first
assertion. We give a correct proof below.

\begin{thm}
\label{cornersFreeze}
Let $X = \Pi_{i=1}^n [r_i,s_i]_{\Z}$.
Let $A = \Pi_{i=1}^n \{r_i,s_i\}$.
\begin{itemize}
\item Let $Y = \Pi_{i=1}^n [a_i,b_i]_{\Z}$ be
      such that $[r_i,s_i] \subset [a_i,b_i]_{\Z}$ for all $i$. 
      Let $f: X \to Y$ be $c_1$-continuous. If 
      $A \subset \Fix(f)$, then $X \subset \Fix(f)$.
\item $A$ is a freezing set for $(X,c_1)$ that is minimal for 
      $n \in \{1,2\}$.
\end{itemize}
\end{thm}

\begin{figure}
    \centering
    \includegraphics{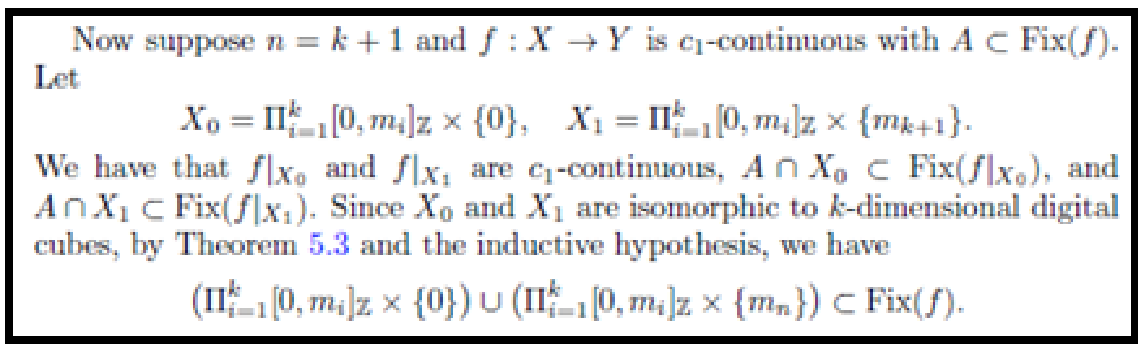}
    \caption{A part of the inductive step of an induction
    argument for Theorem~\ref{cornersFreeze} 
    in~\cite{BxFpSets2}, containing an error. Since
    it has not been shown that $f(X_0) \subset X_0$ and
    $f(X_1) \subset X_1$, it does not follow that
    $A \cap X_0 \subset \Fix(f|_{X_0})$ and
    $A \cap X_1 \subset \Fix(f|_{X_1})$.
    }
    \label{fig:errorInInductiveStep}
\end{figure}

As shown in Figure~\ref{fig:errorInInductiveStep},
it was not shown in~\cite{BxFpSets2} for the inductive
step of the argument offered as proof
that $f(X_0) \subset X_0$ and $f(X_1) \subset X_1$; hence if
$X$ is a proper subset of $Y$ then it does not follow that
    $A \cap X_0 \subset \Fix(f|_{X_0})$ and
    $A \cap X_1 \subset \Fix(f|_{X_1})$.
In the following, we give a correct proof of
the first assertion of Theorem~\ref{cornersFreeze}, using a rather
different approach than was employed in~\cite{BxFpSets2}.

\begin{proof}
By Theorem~\ref{freezeInvariant}, we may assume
\[ X = \Pi_{i=1}^n [0,m_i]_{\Z},~~~~ A = \{ \,0,m_i \, \}.
\]
Let $f \in C(X,c_1)$ such that $f|_A = \id_A$. 
Observe that by Proposition~\ref{uniqueShortestProp}, 
\begin{equation}
\label{edgeInFix(f)}
\begin{array}{l}
\mbox{if $b_1$ and $b_2$ are members of $\Fix(f)$ that differ in 
exactly one coordinate,} \\
\mbox{then the digital segment from $b_1$ 
to $b_2$ is a subset of $\Fix(f)$.}
\end{array}
\end{equation}
In particular, let
\[ X_1 = \left \{ \begin{array}{ll} x \in X \mid & x \mbox{ belongs to a 1-cube (an axis-parallel segment) } \\
  & \mbox{with endpoints in } A \end{array} \right \}.
\]
By~(\ref{edgeInFix(f)}), $X_1 \subset \Fix(f)$.

We proceed inductively. For $j \in \{\, 1, \ldots, n \, \}$, let
\[ X_j = \{ \, x \in X \mid x \mbox{ belongs to a $j$-face of } X \, \}.
\]
Note a $j$-face of $X$ is a $j$-cube with corners in $A$.
Suppose $X_{\ell} \subset \Fix(f)$ for some $\ell \ge 1$.
Given $x \in X_{\ell +1}$, let $K$ be an $(\ell + 1)$-face 
of $Y$ such that $x \in K$. 
Let $F_1$ and $F_2$ be opposite faces of $K$, i.e., for some index~$d$,
$x_i \in F_i$ implies $p_d(x_1)=0$ and $p_d(x_2)=m_d$.

Let $x \in K$. Then $x$ is a point of an axis-parallel segment
from a point of $F_1$ to a point of $F_2$.
By~(\ref{edgeInFix(f)}), $x \in \Fix(f)$.
Thus, $X_{\ell+1} \subset \Fix(f)$. 
This completes our induction. In particular, 
$X = X_n \subset \Fix(f)$.

Thus for $Y=X$, it follows that $A$ is a freezing set for
$(X,c_1)$. That $A$ is minimal for $n \in \{\, 1,2 \, \}$
follows as in~\cite{BxFpSets2}.
\end{proof}

The set of corners of a cube is not always a minimal $c_1$-freezing set,
as shown by the following example in which the set $A$ is a proper
subset of the set of corners.

\begin{exl}
{\rm \cite{BxFpSets2}}
\label{hypercubeEx}
Let $X = [0,1]_{\Z}^3$. Let
\[ A = \{ \, (0, 0, 0), (0, 1, 1), (1, 0, 1), (1, 1, 0) \, \}.
\]
Then $A$  is a  freezing set for $(X, c_1)$.
\end{exl}

\section{$c_1$-Freezing sets for unions of cubes}

In this section, we show how to obtain $c_1$-freezing sets for 
finite subsets of $\Z^n$.

\begin{thm}
\label{unionCubes}
Let $X = \bigcup_{i=1}^m K_i$ where $m \ge 1$,
\[ K_i = \prod_{j=1}^n [a_{ij}, b_{ij}]_{\Z} \subset \Z^n,
\]
and $X$ is $c_1$-connected. Let 
\[ A_i = \prod_{j=1}^n \{ \, a_{ij},b_{ij} \, \}.
\]
Let $A = \bigcup_{i=1}^n A_i$.
Then $A$ is a freezing set for $(X,c_1)$.
\end{thm}

\begin{proof}
Let $f \in C(X,c_1)$ be such that $A \subset \Fix(f)$.

Given $x \in X$, we have $x \in K_i$ for some $i$. Since 
$A_i \subset \Fix(f)$, it follows from Theorem~\ref{cornersFreeze}
that $x \in \Fix(f)$. Thus $X = \Fix(f)$, so $A$ is a freezing set.
\end{proof}

\begin{cor}
The wedge $K_1 \vee K_2$ of two digital cubes in $\Z^n$ with axis-parallel
edges has for a $c_1$-freezing set $K_1' \cup K_2'$, where
$K_i'$ is the set of corners of $K_i$.
\end{cor}

\begin{proof}
This follows immediately from Theorem~(\ref{unionCubes}).
\end{proof}

\begin{remark}
Theorem~\ref{unionCubes} can be used to obtain a $c_1$-freezing
set for any finite digital image $X \subset \Z^3$, since $X$ is
trivially a union of cubes
$[a,a]_{\Z} \times  [b,b]_{\Z} \times [c,c]_{\Z}$.
\end{remark}

\begin{remark}
Often, the freezing set of Theorem~\ref{unionCubes} is not
minimal. However, the theorem is valuable in that it often
gives a much smaller subset of $X$ than $X$ itself as a freezing
set. As a simple example of the non-minimal assertion, consider
\[ X = [0,4]_{\Z}^2 \times [0,2]_{\Z} \cup  [0,4]_{\Z}^2 \times [2,4]_{\Z}.
\]
For this description of $X$, Theorem~\ref{unionCubes} gives the
$c_1$-freezing set
\[ A = \{ \, 0,4 \, \}^2 \times \{ \, 0,2 \, \} \cup
        \{ \, 0,4 \, \}^2 \times \{ \, 2,4 \, \} =
         \{ \, 0,4 \, \}^2 \times \{ \, 0,2, 4 \, \},
\]
a set of 12 points.
However, by observing that $X$ can be described as
$X = [0,4]_{\Z}^3$, we obtain from Theorem~\ref{cornersFreeze}
the $c_1$-freezing set $A' = \{ \, 0,4 \, \}^3$, a set of 8 points.
\end{remark}

The following example shows that a cubical ``cavity" 
(see Figure~\ref{fig:cubeWithCavity}) need not affect
determination of a freezing set.

\begin{figure}
    \centering
    \includegraphics{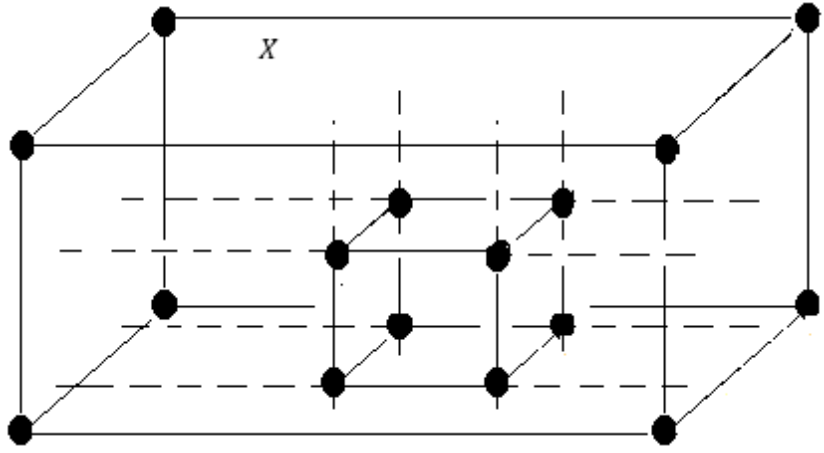}
    \caption{A cube with a cubical cavity}
    \label{fig:cubeWithCavity}
\end{figure}

\begin{exl}
Let $X = [0,6]_{\Z}^3 \setminus [2,4]_{\Z}^3$. Then
$A = \{ \, 0,6 \, \}^3$ is a freezing set for $(X,c_1)$.
\end{exl}

\begin{proof}
Note that by Theorem~\ref{cornersFreeze}, $A$ is a freezing
set for $K=([0,6]_{\Z}^3,c_1)$. We show that removing
$[2,4]_{\Z}^3$ need not change the freezing set.

Observe we can decompose $X$ as a union of cubes as follows:
Let 
\[ L_X \mbox{ (left) } = [0,6]_Z \times [0,1]_{\Z} \times [0,6]_Z,~~~
   R_X \mbox{ (right) } = [0,6]_Z \times [5,6]_{\Z} \times [0,6]_Z,
\]
\[ F_X \mbox{ (front) } = [5,6]_{\Z} \times [0,6]_Z^2,
   ~~~Ba_X \mbox{ (back) } = [0,1]_{\Z} \times [0,6]_Z^2 ,
\] 
\[ Bo_X \mbox{ (bottom) } = [0,6]_Z^2 \times [0,1]_{\Z},~~~
   T_X \mbox{ (top) } = [0,6]_Z^2 \times [5,6]_{\Z}.
\]
Then $X = L_X \cup R_X \cup F_X \cup Ba_X \cup Bo_X \cup T_X$.
Theorem~\ref{cornersFreeze} gives us a freezing set $B$ for
$(X,c_1)$ consisting of the corners of each of
$L_X, R_X, F_X, Ba_X, Bo_X, T_X$.

However, suppose $f \in C(X,c_1)$ is such that $f|_A = \id_A$.
As in the proof of Theorem~\ref{unionCubes}, each of the faces
$L,R,F,Ba,Bo,T$ of $K$ is a subset of $\Fix(f)$. Therefore, each
$x \in B \setminus A$ is on an axis-parallel digital segment that joins two
points of one of $L,R,F,Ba,Bo,T$, so by Proposition~\ref{uniqueShortestProp},
$x \in \Fix(f)$. Therefore, $A$ is a freezing set for $(X,c_1)$.
\end{proof}

\section{$c_n$-freezing sets in $\Z^n$}
We have the following.

\begin{prop}
\label{BdForc_u}
{\rm \cite{BxFpSets2}}
Let $X$ be a finite digital image in $\Z^n$. Let $A \subset X$.
Let $f \in C(X,c_u)$, where $1 \le u \le n$. If
$Bd(A) \subset \Fix(f)$, then $A \subset \Fix(f)$.
\end{prop}

\begin{thm}
\label{BdFreezes}
{\rm \cite{BxFpSets2}}
Let $X$ be a finite digital image in $\Z^n$. For $1 \le u \le n$, $Bd(X)$ is a
freezing set for $(X, c_u)$.
\end{thm}

The following is inspired by Theorem~\ref{BdFreezes}.

\begin{thm}
Let $X = \prod_{i=1}^n [a_i,b_i]_{\Z}$, where for all~$i$, $b_i > a_i$.
Then $Bd(X)$ is a minimal freezing set for $(X,c_n)$.
\end{thm}

\begin{proof}
By Theorem~\ref{BdFreezes}, $Bd(X)$ is a freezing set for $(X,c_n)$.
We must show its minimality.

Consider a point 
$x_0 = (x_1, \ldots, x_n) \in Bd(X) = \prod_{i=1}^n \{ \, a_i,b_i \, \}$.
For some index~$i$, $p_i(x_0) \in \{ \, a_i,b_i \, \}$. Without loss
of generality, $p_i(x_0)=a_i$. Consider the function
$f: X \to X$ given by
\[ f(x) = \left \{ \begin{array}{ll}
     x & \mbox{if } x \neq x_0; \\
     (x_1, \ldots, x_{i-1}, a_i + 1, x_i, \ldots, x_n) & \mbox{if } x = x_0.
\end{array} \right .
\]
Suppose $x' \adj_{c_n} x$ in $X$.
\begin{itemize}
    \item If neither $x$ nor $x'$ is $x_0$, then $f(x') = x' \adj x = f(x)$.
    \item Suppose, say, $x=x_0$.
    \begin{itemize}
        \item If $p_i(x') \neq a_i$ then $x'$ differs from $x_0$ in at most~$n$
          coordinates, including having $p_i(x')=a_i+1$, so
          $f(x')=x'$ differs from $f(x)$ by 1 in at most~$n-1$ coordinates,
          and by 0 in any other coordinates.
        \item If $p_i(x') = a_i$ then $x'$ differs from $x$ by 1 in at most~$n-1$
          coordinates, so $f(x')=x'$ differs from $f(x)$ by 1 in at most~$n$ coordinates, and by 0 in any other coordinates.
    \end{itemize}
\end{itemize}
Thus $f \in C(X,c_n)$, 
$f|_{Bd(X) \setminus \{\, x \, \}} = \id_{Bd(X) \setminus \{\, x \, \}}$,
and $f \neq \id_X$. Thus ${Bd(X) \setminus \{\, x \, \}}$ is not a
$c_n$-freezing set for $X$. Thus $Bd(X)$ is a minimal freezing set.
\end{proof}

\begin{thm}
Let $X = \bigcup_{i=1}^m K_i$ where $m \ge 1$,
\[ K_i = \prod_{j=1}^n [a_{ij}, b_{ij}]_{\Z} \subset \Z^n,
\]
and $X$ is $c_n$-connected. Let $A_i = Bd(K_i)$ and let
$A = \bigcup_{i=1}^m A_i$. Then $A$ is a freezing set for $(X,c_n)$.
\end{thm}

\begin{proof}
By Theorem~\ref{BdFreezes}, $A_i$ is a freezing set for $(K_i,c_n)$. Let
$f \in C(X,c_n)$ be such that $f|_A = \id_A$. Given $x \in X$, for some
index~$q$ we have $x \in K_q$. It follows from 
Proposition~\ref{BdForc_u} that $K_q \subset \Fix(f)$.
Thus $f=\id_X$, and the assertion follows.
\end{proof}

\section{Further remarks}
We have studied freezing sets for finite digital images
in~$\Z^n$ with respect to the $c_1$ and $c_n$ adjacencies.
For both of these adjacencies, we have shown that a
decomposition of an image $X$ as a finite union of cubes lets us
find a freezing set for $X$ as a union of freezing sets
for the cubes of the decomposition. Such a freezing set
is not generally minimal, but often is useful in having
cardinality much smaller than the cardinality of~$X$.

\end{document}